\documentclass[12pt]{article}
\usepackage[lofdepth,lotdepth,caption=false]{subfig}
\usepackage{fancyhdr}
\usepackage{hyperref}
\usepackage{amsmath, amssymb, graphicx,amsthm}
\usepackage{xspace}
\usepackage{braket}
\usepackage{color}
\usepackage{setspace}
\usepackage{amscd}
\usepackage{amsxtra}
\usepackage[mathscr]{eucal}
\usepackage{mathrsfs}
\usepackage{upgreek}
\usepackage{indentfirst}

\def\C{\mbb{C}}

\newcommand{\be}{\begin{equation}}
\newcommand{\ee}{\end{equation}}
\newcommand{\mbb}[1]{\mathbb{#1}}

\numberwithin{equation}{section}

\usepackage{thmtools}
\usepackage{thm-restate}

\usepackage{hyperref}

\usepackage{cleveref}

\newtheorem{thm}{Theorem}[section]
\newtheorem{defi}[thm]{Definition}
\newtheorem{prop}[thm]{Propostion}

\newtheorem{rmk}[thm]{Remark}
\newtheorem{lem}[thm]{Lemma}

\newcommand{\abs}[1]{\left\lvert#1\right\rvert}
\newcommand{\norm}[1]{\lVert#1\rVert}
\newcommand{\inpro}[1]{\langle#1\rangle}
\newcommand{\inppro}[1]{\left(#1\right)}
\newcommand{\sfl}[1]{\mathrm{sf}\{#1\}}

\newcommand{\tr}[1]{\mathrm{tr}\left[#1\right]}

\def\D{\mathscr{D}}

\def\d{\mathrm{d}}
\def\ch{\mathrm{ch}}

\def\End{\mathrm{End}}

\def\Sp{\mathbb{S}}

\def\D2{(\widehat{D}^E_{s,\varepsilon})^2}

\def\de2{(D^{E}_s)^2}

\begin{document}
\title{Spectral Flow, Llarull's Rigidity Theorem in Odd Dimensions and its Generalization}
\date{}
\author{Yihan Li\footnote{
	Chern Institute of Mathematics \& LPMC, 
	Nankai University,
	Tianjin 300071,
	P.R.China.  yhli@nankai.edu.cn} , Guangxiang Su \footnote{
	Chern Institute of Mathematics \& LPMC, 
	Nankai University,
	Tianjin 300071,
	P.R.China.   guangxiangsu@nankai.edu.cn}and Xiangsheng Wang\footnote{
	School of Mathematical Sciences, Shandong University, Jinan, Shandong 250100, P.R.China.  xiangsheng@sdu.edu.cn}}
\maketitle

\begin{abstract}
For a compact spin Riemannian manifold $(M,g^{TM})$ of dimension $n$ such that the associated scalar curvature $k^{TM}$ verifies that $k^{TM}\geqslant n(n-1)$, Llarull's rigidity theorem says that any area-decreasing smooth map $f$ from $M$ to the unit sphere $\Sp^{n}$ of nonzero degree is an isometry. We present in this paper a new proof for Llarull's rigidity theorem in odd dimensions via a spectral flow argument. This approach also works for a generalization of Llarrull's theorem when the sphere $\Sp^{n}$ is replaced by an arbitrary smooth strictly convex closed hypersurface in $\mathbb{R}^{n+1}$. The results answer two questions by Gromov.
\end{abstract}


\section{Introduction}
It is well-known that starting with the Lichnerowicz vanishing theorem, Dirac operators have played important roles in the study of Riemannian metrics of positive scalar curvature on spin manifolds. A notable example is Llarull's rigidity theorem \cite{llarull1998sharp}, which states as follows.

\begin{thm}[Llarull]\label{Llarull}
Let $M$ be a closed spin manifold of dimension $n$ $(\geqslant 2)$, equipped with a Riemannian metric $g^{TM}$. Suppose that there exists a $(1,\Lambda^2)$-contracting map $f$ from $(M,g^{TM})$ to the unit sphere $(\mathbb{S}^{n}, g_0)$ of non-zero degree. Then either there exists a point $x\in M$ where the scalar curvature $k^{TM}(x)< n(n-1)$, or $f$ is an isometry.
\end{thm}

When the dimension $n$ is odd, the classical methods to prove Llarull's theorem need to convert the problem to even dimensions and involve $\Sp^{n+1}$ \cite[pp. 133-134]{gromov2019four}. Gromov asks in \cite[pp. 134]{gromov2019four} whether there is a more direct proof of Llarull's theorem for odd dimension $n$ without direct reference to $\Sp^{n+1}$, and, in particular, by a spectral flow argument. The first part of this paper provides a positive answer to this question. As in \cite{llarull1998sharp}, the theorem is proved by contradiction. However, in odd dimensions, the contradiction is accomplished via the spectral flow of a family of (twisted) Dirac operators rather than the index of a single Dirac operator.

Furthermore, as indicated in \cite{gromov2019four}, the method to prove Theorem \ref{Llarull} can indeed be applied straightforwardly to show the following ``spin-area convex extremality theorem" in odd dimensions, which answers another question in \cite[pp. 131, Question (b)]{gromov2019four} and compares to the result of Goette-Semmelmann \cite{goette2002scalar} and Listing \cite{listing2010scalar} in even dimensions.

\begin{thm}\label{convexsfc}
Let $M$ be a closed spin manifold of odd dimension $2k-1$ $(k\geqslant 2)$ equipped with a Riemannian metric $g^{TM}$, and $X\subset \mathbb{R}^{2k}$ be a smooth strictly convex closed hypersurface equipped with the metric $g_0$ induced by the Euclidean metric in $\mathbb{R}^{2k}$. Suppose that there exists a $(1,\Lambda^2)$-contracting map $f: (M,g^{TM})\rightarrow (X, g_0)$ of non-zero degree, then either there exists a point $x\in M$ where the scalar curvature $k^{TM}(x)< k^{TX}(f(x))$, or $f$ is an isometry.
\end{thm}

This paper is organized as follows. In Section 2, we briefly review some preliminaries on the Clifford algebra and the Hermitian space of complex spinors, and in particular, construct a family of unitary connections on a trivial vector bundle over the unit sphere $\Sp^{2k-1}$ as in \cite{getzler1993odd}. Then, in Section 3, we provide a direct proof for Llarull's rigidity theorem in odd dimensions via spectral flow in the following procedure. Based on the family of unitary connections, we construct a family $\{D_s^E:s \in [0,1]\}$ of twisted Dirac operators on the manifold $M$ whose spectral flow $\sfl{D_s^E:s\in [0,1]}$ can be computed explicitly, while on the other hand, the geometric conditions guarantee their invertibility. Finally, in Section 4, we prove Theorem \ref{convexsfc} by adapting the method in previous sections to the case when the sphere $\Sp^{2k-1}$ is replaced by a smooth strictly convex closed hypersurface $X\subset \mathbb{R}^{2k}$.

\section{Algebraic preliminaries}
In this section, we briefly review the algebraic preliminaries that will be involved throughout this paper.
\subsection{The exterior algebra}
Let $(V, \inpro{\cdot,\cdot})$ be a real vector space endowed with an inner product $\inpro{\cdot,\cdot}$. In the dual space $V^\ast$ of $V$, given any $\alpha^1, \alpha^2 \in V^\ast$, by the Riesz representation, there exist uniquely $u_{1}, u_2\in V$, such that for all $v\in V$, $\alpha_i(v)=\inpro{u_i,v}$, $i=1,2$, and the inner product on $V^\ast$ is defined by
\[
\inpro{\alpha^1,\alpha^2}=\inpro{u_1,u_2}.
\]
The inner product $\inpro{\cdot,\cdot}$ on $V^\ast$ can be extended to the exterior algebra $\Lambda^\ast (V^\ast)$ in the following way. Firstly, given any monomials $\omega_1=\alpha^1\wedge \alpha^2 \wedge \cdots \wedge \alpha^p$ and $\omega_2=\beta^1\wedge \beta^2 \wedge \cdots \wedge \beta^p$ of the same degree $p$, we define
\begin{equation}\label{metric2form}
\inpro{\omega_1, \omega_2}=\det({\inpro{\alpha^i,\beta^j}}),
\end{equation}
and extend it linearly to $\Lambda^\ast (V^\ast)$ via the direct sum $\Lambda^\ast (V^\ast)=\bigoplus_{p} \Lambda^p (V^\ast)$. And the induced norm on $\Lambda^\ast (V^\ast)$ is determined by $\abs{\omega}=\inpro{\omega,\omega}^{\frac{1}{2}}$.

Note that, on a Riemannian manifold $(M,g^{TM})$, the metric $g^{TM}$ induces a pointwise norm $\abs{\cdot}$ for differential forms in the same way. 

\begin{defi}
A smooth map $f :M \rightarrow N$ between Riemannian manifolds $M$ and $N$ is called $(\varepsilon, \Lambda^p)$-contracting $(\varepsilon>0)$ with respect to the norm $\abs{\cdot}$ if
\begin{equation}\label{pcontracting}
\abs{f^\ast \omega}\leqslant \varepsilon \abs{\omega}
\end{equation}
holds for all $\omega\in \Omega^p(N)$.
\end{defi}

\subsection{The space of spinors}
Given a real vector space $V$ with an inner product $\inpro{\cdot,\cdot}$, the Clifford algebra $Cl(V, \inpro{\cdot,\cdot})$ is generated over $\mathbb{R}$ by $V$ with the commutation relations
\[
v \cdot w+w\cdot v=-2\inpro{v,w}.
\]
In particular, let $\mathbb{R}^{n}$ be the canonical oriented $n$-dimensional Euclidean space with the canonical inner product, which generates the Clifford algebra $Cl(\mathbb{R}^{n})$. Let $\bar{c}$ denote the action of the (complexified) Clifford algebra $\C l(\mathbb{R}^{n}):= Cl(\mathbb{R}^n) \otimes_{\mathbb{R}} \mathbb{C}$ on $S_{n}$\footnote{The notation $\bar{c}(\cdot)$ is intentionally used here throughout this section to distinguish it from the action $c(\cdot)$ of the Clifford bundle $\mathbb{C}l(TM)$ on the spinor bundle $S(TM)$ in the following sections.}, a complex Hermitian space of spinors whose dimension is $2^{[\frac{n}{2}]}$. We equip $S_{n}$ with the Hermitian metric $\inpro{\cdot, \cdot}_{S_n}$ such that any vector of unit length in $\mathbb{R}^{n}$ acts on $S_{n}$ isometrically\footnote{The existence of such a metric is guaranteed by \cite[Proposition I.5.16]{lawson2016spin}.}. Let $\{\partial_1 ..., \partial_{n}\}$ be the canonical oriented orthonormal basis of $\mathbb{R}^{n}$. For convenience, we denote $\bar{c}(\partial_i)$ by $\bar{c}_i$.

When $n=2k$, by \cite{ABS}, $\C l(\mathbb{R}^{2k})$ identifies with $\End(S_{2k})$. In addition, we set $\tau = (-1)^{k} \bar{c}_1 \bar{c}_2 \cdots \bar{c}_{2k}$. Then $\tau^2= 1$. Define $S_{2k,\pm} = \{s\in S_{2k}; \tau s = \pm s\}$. Then $S_{\pm, 2k}$ has dimension $2^{k-1}$, and $S_{2k} =S_{2k,+}\oplus S_{2k,-}$. Furthermore, given any vector $\mathbf{v}\in \mathbb{R}^{2k}$, 
\[
\bar{c}(\mathbf{v}) (S_{2k,\pm})\subseteq S_{2k,\mp}.
\]

Another thing that needs to be stated here is the Clifford action of $2$-forms $\alpha\in \Lambda^2(\mathbb{R}^n)$. In terms of the orthonormal basis $\{\partial_i\}_{i=1}^{n}$, its Clifford action is determined by
\begin{equation}\label{clff2form}
\bar{c}(\alpha):=\frac{1}{2}\sum_{i,j=1}^{n} \alpha(\partial_i,\partial_j)\bar{c}_i\bar{c}_j=\sum_{1\leqslant i<j\leqslant n}\alpha(\partial_i,\partial_j)\bar{c}_i 
\bar{c}_j.
\end{equation}
On a Riemannian manifold, the Clifford action of a $2$-form can be determined similarly in terms of any local orthonormal frame of its tangent bundle.

\subsection{A trivial vector bundle over the sphere $\Sp^{2k-1}$}

In the Euclidean space $\mathbb{R}^{2k}$ equipped with the canonical metric, we regard the $(2k-1)$-dimensional unit sphere $(\mathbb{S}^{2k-1},g_0)$ as a subset of the $\mathbb{R}^{2k}$, so that $g_0$ is the naturally induced metric. For any point $x\in \mathbb{S}^{2k-1}\subset \mathbb{R}^{2k}$, we write $x=(x^1,x^2,\dots, x^{2k})$, and let $E_0$ be the trivial bundle $\Sp^{2k-1}\times S_{2k,+}$. Then the space $C^\infty(M, E_0)$ of smooth sections of $E_0$ can be identified with the space $C^\infty(\Sp^{2k-1}, S_{2k,+})$ of $S_{2k,+}$-valued smooth functions on $\Sp^{2k-1}$. And we equip $E_0$ with a trivial connection $\d$ such that for any smooth function $f\in C^\infty(M)$ and $\psi\in S_{2k,+}$,
\[
\d (f\cdot \psi)=\d f \otimes \psi.
\]

Here we consider a function $g: \Sp^{2k-1} \rightarrow \End(S_{2k,+})$ that is defined in \cite{getzler1993odd} by
\begin{equation}
g(x)=\bar{c}_{2k}\bar{c}(x)=\sum_{i=1}^{2k} x^i \cdot \bar{c}_{2k}\bar{c}_i.
\end{equation}
Let $g$ act as a smooth automorphism of $E_0$ and then we define $\omega \in \Omega^1(\Sp^{2k-1},\End(E_0))$ by $\omega=g^{-1} [\d,g]$. By direct computation, it is determined by
\begin{equation}\label{cntform}
\omega(x)=\bar{c}(x) \bar{c}_{2k} \bar{c}_{2k} \sum_{i=1}^{2k} \d x^i\otimes \bar{c}_i=-\sum_{i=1}^{2k}\d x^i \otimes \bar{c}(x) \bar{c}_i, \qquad \forall x\in \Sp^{2k-1}.
\end{equation}

Now we define a family $\{\nabla_{s}:s\in [0,1]\}$ of connections on $E_{0}$ by
\begin{equation}\label{cnntsonsph}
\nabla_{s}=\d+s g^{-1} [\d,g]=\d+s\omega.
\end{equation}
We denote by $\inpro{\cdot, \cdot}_{S_{2k,+}}$ the restriction of the Hermitian metric $\inpro{\cdot, \cdot}_{S_{2k}}$ on $S_{2k,+}$ and extend it trivially to be a Hermitian metric $\inpro{\cdot, \cdot}_{E_0}$ on the bundle $E_0$. The key point here is that the construction \eqref{cnntsonsph} actually defines a family of unitary connections on the vector bundle $E_0$.

\begin{lem}\label{unitarycnct}
For all $s\in[0,1]$, $\nabla_{s}$ is compatible with the metric $\inpro{\cdot, \cdot}_{E_0}$ on $E_0$, i.e. given any $\phi,\psi \in C^\infty(\Sp^{2k-1}, S_{2k,+})$,
\[
\d \inpro{\phi,\psi}_{E_0}=\inpro{\nabla_{s}\phi, \phi}_{E_0}+\inpro{\phi, \nabla_{s}\psi}_{E_0}.
\]
\end{lem}
\begin{proof}
For the trivial connection $\nabla_0=\d$, it follows from the Leibniz rule that
\[
\d\inpro{\phi,\psi}_{E_0}=\inpro{\d \phi, \psi}_{E_0}+\inpro{\phi,\d  \psi}_{E_0}.
\]
On the other hand, note that the Clifford action of any vector of unit length in $\mathbb{R}^{2k}$ is an isometry on $(S_{2k},\inpro{\cdot,\cdot}_{S_{2k}})$, which implies that, for all $x\in \Sp^{2k-1}$, $g (x)$ is an isometry on $(S_{2k,+},\inpro{\cdot,\cdot}_{S_{2k,+}})$. Therefore, for the connection $\nabla_{1}=g^{-1}\circ \d\circ g$, one has
\[
\begin{split}
&\inpro{(g^{-1}\circ\d \circ g) \phi, \psi}_{E_0}+\inpro{\phi,(g^{-1}\circ \d \circ g) \psi}_{E_0}\\
=&\inpro{\d(g \phi), g\psi}_{E_0}+\inpro{g\phi,\d (g  \psi)}_{E_0}\\
=&\d\inpro{g\phi, g\psi}_{E_0}=\d\inpro{\phi,\psi}_{E_0}.
\end{split}
\]
Then, for all $s\in [0,1]$, noticing that $\nabla_{s}=(1-s)\d + s \nabla_{1}$, it follows that
\[
\begin{split}
\inpro{\nabla_{s} \phi, \psi}_{E_0}+\inpro{\phi,\nabla_{s}  \psi}_{E_0}=&(1-s)\left( \inpro{\d \phi, \psi}_{E_0}+\inpro{\phi,\d \psi}_{E_0}\right)+s\left(\inpro{\nabla_{1} \phi, \psi}_{E_0}+\inpro{\phi,\nabla_{1}  \psi}_{E_0} \right)\\
=&[s+(1-s)]\d \inpro{\phi,\psi}_{E_0}\\
=&\d \inpro{\phi,\psi}_{E_0}.
\end{split}
\]
\end{proof}

Let $F_s$ be the curvature operator $\nabla_s^2$, and it is given in \cite{getzler1993odd} by the formula
\[
F_s=-s(1-s)\omega^2.
\]
At any point $x\in \Sp^{2k-1} \subset \mathbb{R}^{2k}$, we identify $T_x \mathbb{R}^{2k}$ with $\mathbb{R}^{2k}$ in the canonical way, so that its subspace $T_x \Sp^{2k-1}$ is identified with the subspace $\{x\}^\perp$ of $\mathbb{R}^{2k}$. Over an open set $U\subset M$, we take a local orthonormal frame $\{\epsilon_i\}_{i=1}^{2k-1}$ of $T\Sp^{2k-1}$. Then, for each $\epsilon_i$ and $x\in U$, we have that
\[
\omega(x)(\epsilon_i)=-\sum_{j=1}^{2k}  \d x^j(\epsilon_i) \cdot \bar{c}(x)  \bar{c}_j=-\sum_{j=1}^{2k}  \inpro{\epsilon_i,\partial_j} \cdot \bar{c}(x)  \bar{c}_j
=-\bar{c}(x)  \bar{c}(\epsilon_i),
\]
and it then follows for all $i\neq j$, that
\[
\begin{split}
\omega^2(x)(\epsilon_i, \epsilon_j)=&\omega(x)(\epsilon_i) \omega(x)(\epsilon_j)-\omega(x)(\epsilon_j)\omega(x)(\epsilon_i)\\
=&\bar{c}(x) \bar{c}(\epsilon_i)\bar{c}(x) \bar{c}(\epsilon_j)-\bar{c}(x) \bar{c}(\epsilon_j)\bar{c}(x) \bar{c}(\epsilon_i)
=2 \bar{c}(\epsilon_i)\bar{c}(\epsilon_j).
\end{split}
\]
Therefore,
\[
F_s(\epsilon_i, \epsilon_j)=-2s(1-s)\bar{c}(\epsilon_i)\bar{c}(\epsilon_j).
\]
So, if we denote by $\{\epsilon^i\}_{i=1}^{2k-1}$ the dual frame of $\{\epsilon_i\}_{i=1}^{2k-1}$, then
\begin{equation}\label{curvfs}
F_s=-2s(1-s)\sum_{i<j}(\epsilon^i\wedge \epsilon^j)\otimes \bar{c}(\epsilon_i)\bar{c}(\epsilon_j).
\end{equation}

\section{Proof of Llarull's theorem in odd dimensions}
In this section, we present a new proof of Llarull's theorem \cite{llarull1998sharp} in odd dimensions via a spectral flow argument. Let $(M,g^{TM})$ be a $(2k-1)$-dimensional closed spin Riemannian manifold, and denote by $k^{TM}$ the associated scalar curvature. Assume that $f: (M,g^{TM}) \to (\Sp^{2k-1},g_0)$ is a smooth $(1,\Lambda^2)$-contracting map of nonzero degree. We start by constructing a family $\{D_s^E: s\in [0,1]\}$ of twisted Dirac operators, whose spectral flow is $\deg (f)\neq 0$ by explicit computation. Then we show that the assumptions on $k^{TM}$ imply the vanishing of spectral flow, which leads to the contradiction. Finally, we show the rigidity result, i.e. $k^{TM}(x)\equiv (2k-1)(2k-2)$ implies that $f$ maps $(M,g^{TM})$ isometrically to the unit sphere $(\Sp^{2k-1},g_0)$.

\subsection{A family of Dirac operators}
On the manifold $M$, we denote by $S$ its spinor bundle $S(TM)$ and equip it with the connection $\nabla^S$ induced by the Levi-Civita connection on the tangent bundle $TM$. In addition, let $E=f^\ast E_{0}$ be the pull-back of the trivial vector bundle $E_{0}$ over the sphere $\Sp^{2k-1}$, and equip it with the pull-back metric. The $\{f^\ast \nabla_s: s\in [0,1]\}$ is a family of unitary connections. And, for each $s\in [0,1]$, the curvature $R^E_s$ associated to $f^\ast\nabla_s$ is determined by
\begin{equation}\label{curvsp}
R^E_s=(f^\ast \nabla_s)^2=f^\ast(\nabla_s)^2=f^\ast F_s.
\end{equation}

From now on, we consider the twisted bundle $S\otimes E$ over the manifold $M$ with the tensor product metric and equip it with a family $\{\nabla^{S\otimes E}_s\}_{s\in [0,1]}$ of unitary connections defined by
\[
\nabla^{S\otimes E}_s= \nabla^S\otimes 1+1\otimes f^\ast \nabla_{s}.
\]
We denote by $c(\cdot)$ the action of the Clifford bundle $\mathbb{C}l(TM)$ on the spinor bundle $S$. Moreover, for any vector field $X\in C^\infty(M,TM)$, we extend $c(X)$ to an action on $S\otimes E$ by acting as identity on $E$, and still denote it by $c(X)$. The above family of unitary connections then induces a family of twisted Dirac operators $D^E_s: C^\infty(M,S\otimes E) \rightarrow C^\infty(M,S\otimes E)$, which, in terms of a local orthonormal frame $\{e_i\}_{i=1}^{2k-1}$ of $TM$, is determined by
\[
D^E_s=\sum_{i=1}^{2k-1} c(e_i) \nabla^{S\otimes E}_{s,e_i}.
\]

For this family of twisted Dirac operators, it follows from the Lichnerowicz formula that
\begin{equation}\label{lichnerowicz}
\de2=\nabla^{S\otimes E,\ast}_s \nabla^{S\otimes E}_s+\frac{k^{TM}}{4}+c(R^{E}_s).
\end{equation} 
Furthermore, in terms of a local orthonormal frame $\{e_i\}_{i=1}^{2k-1}$ of $TM$, the connection Laplacian $\nabla^{S\otimes E,\ast}_s \nabla^{S\otimes E}_s$ has the expression as
\[
\nabla^{S\otimes E,\ast}_s \nabla^{S\otimes E}_s=-\sum_{i=1}^{2k-1} \nabla^{S\otimes E}_{s,e_i} \nabla^{S\otimes E}_{s,e_i}+\nabla^{S\otimes E}_{s, \nabla^{TM}_{e_i} e_i},
\]
and $c(R^{E}_s)$ is determined by
\begin{equation}\label{clif2fm}
c(R^{E}_s)=\sum_{i<j} c(e_i)c(e_j) \otimes R^{E}_s(e_i, e_j).
\end{equation}

\subsection{Spectral flow of the family $\{D^E_{s}: s\in [0,1]\}$}
Before computing the spectral flow of the family $\{D_s^E, s\in [0,1]\}$, we recall the definition of spectral flow for a family of self-adjoint Fredholm operators in \cite{APS3}.
\begin{defi}\label{spectralflow}
For $\{D^E_s: s\in [0,1]\}$, as a smooth family of self-adjoint Fredholm operators, the spectral flow $\sfl{D_s^E: s\in [0,1]}$, which is also denoted by $\sfl{D_s^E}$ for short, counts the net number of eigenvalues of $D_s^E$ which change sign when the parameter $s$ varies from $0$ to $1$.
\end{defi}
\begin{rmk}
Based on this definition, if every single operator $D_s^E$ in such a family is invertible, there cannot be any eigenvalue changing its sign when the parameter $s$ varies from $0$ to $1$, and therefore, the spectral flow $\sfl{D_s^E}$ is forced to vanish.
\end{rmk}
The spectral flow is computed directly with the formula from \cite[Theorem 2.8]{getzler1993odd} as follows.
\begin{prop}
The spectral flow $\sfl{D^E_s}$ is given by
\begin{equation}\label{sf}
\sfl{D^E_{s}}=-\deg(f).
\end{equation}

\end{prop}
\begin{proof}
Firstly, by \cite[Theorem 2.8]{getzler1993odd}, the spectral flow is given by the formula
\[
\sfl{D^E_s}=-\left(\frac{1}{2\pi \sqrt{-1}}\right)^k\int_M \hat{A}(TM,\nabla^{TM}) f^\ast \ch(g,\d),
\]
where $\ch(g,\d)\in \Omega^\textrm{odd} (\Sp^{2k-1})$ is the odd Chern character form defined by
\[
\ch(g,\d)=\sum_{n=0}^{k-1}\frac{n!}{(2n+1)!}\tr{(g^{-1}\d g)^{2n+1}}.
\]
Note that $\ch(g,\d)$ is a closed differential form (c.f. \cite{zhang2001lectures}) on the sphere $\Sp^{2k-1}$ and that every closed form on the sphere $\Sp^{2k-1}$ of a positive degree strictly less than $(2k-1)$ is exact. It then follows that
\[
\begin{split}
\int_M \hat{A}(TM,\nabla^{TM}) f^\ast \ch(g,\d)=\int_M f^\ast [\ch(g,\d)]^{\textrm{max}}
=\deg(f)\int_{\Sp^{2k-1}}\ch(g,\d),
\end{split}
\]
where $[\ch(g,\d)]^{\textrm{max}}$ denotes the top degree term of $\ch(g,\d)$.

Furthermore, according to the computation in \cite{getzler1993odd}, the integral
\[
-\left(\frac{1}{2\pi \sqrt{-1}}\right)^k \int_{\Sp^{2k-1}}\ch(g,\d)=-1,
\]
and therefore
\[
\sfl{D^E_{s}}=-\deg(f).
\]
\end{proof}

\subsection{Vanishing of the spectral flow}
In this section, we assume that the scalar curvature $k^{TM}\geqslant (2k-1)(2k-2)$ all over the manifold $M$ and $k^{TM}(x)> (2k-1)(2k-2)$ at some $x\in M$, and derive from this assumption an estimate of the ``twisted curvature" $c(R^{E}_s)$, which implies the invertibility of the operator $D_s^E$ for all $s\in [0,1]$.

\begin{thm}\label{invsph}
Assume that the scalar curvature $k^{TM}\geqslant (2k-1)(2k-2)$ all over the manifold $M$ and $k^{TM}(x)> (2k-1)(2k-2)$ at some $x\in M$, then the operator $D^E_{s}$ is invertible for all $s\in [0,1]$.
\end{thm}
\begin{proof}
Let $\inppro{\cdot,\cdot}$ be the inner product on the space $C^\infty(M,S\otimes E)$ defined by
\[
\inppro{\phi,\psi}=\int_M\inpro{\phi(x),\psi(x)}\d x, \qquad \forall \phi,\psi \in C^\infty(M,S\otimes E).
\]
Then, from \eqref{lichnerowicz}, it follows that
\begin{equation}\label{lch}
\inppro{\de2 \phi, \phi}=\inppro{\nabla^{S\otimes E,\ast}_s \nabla^{S\otimes E}_s \phi,\phi}+\frac{1}{4}\inppro{k^{TM}\phi,\phi}+\inppro{c(R^{E}_s)\phi,\phi}.
\end{equation}

Given any $x\in M$, we take an orthonormal basis $\{\epsilon_i\}_{i=1}^{2k-1}$ of $T_{f(x)}\Sp^{2k-1}$ and denote by $\{\epsilon^i\}_{i=1}^{2k-1}$ its dual basis. It then follows from \eqref{curvfs} and \eqref{curvsp} that, in terms of this basis,
\begin{equation}\label{rssph}
R^E_s(x)=f^\ast [F_s(f(x))]= -2s(1-s)\sum_{i<j} f^\ast(\epsilon^i \wedge \epsilon^j) \otimes f^\ast [\bar{c}(\epsilon_i) \bar{c}(\epsilon_j)],
\end{equation}
where $\bar{c}(\epsilon_i)\bar{c}(\epsilon_j)$ is regarded as an element in $\End(E_{0,f(x)})$ and thus $f^\ast [\bar{c}(\epsilon_i)\bar{c}(\epsilon_j)]$ acts as an endormorphism of $E_x$. In addition, since $\bar{c}(\epsilon_i)\bar{c}(\epsilon_j)$ is an isometry of $E_{0,f(x)} \cong S_{2k,+}$, $f^\ast [\bar{c}(\epsilon_i)\bar{c}(\epsilon_j)]$ is an isometry of $E_x$.

Furthermore, let $\{e_i\}_{i=1}^{2k-1}$ be an orthonormal basis of $T_xM$, and then, combining with \eqref{clif2fm}, the formula \eqref{rssph} implies that
\begin{equation}\label{cfsph}
\begin{split}
c(R^E_s)\phi(x)=&-2s(1-s)\sum_{ k<l}\sum_{i<j}f^\ast(\epsilon^i\wedge \epsilon^j)(e_k,e_l)\cdot c(e_k)c(e_l)\otimes f^\ast[\bar{c}(\epsilon_i)\bar{c}(\epsilon_j)]\phi(x)\\
=&-2s(1-s)\sum_{i<j}c(f^\ast(\epsilon^i \wedge \epsilon^j))\otimes f^\ast [\bar{c}(\epsilon_i)\bar{c}(\epsilon_j)]\phi(x),
\end{split}
\end{equation}
where the Clifford action of 2-forms is defined as in \eqref{clff2form}. It follows from the result to be shown later in Proposition \ref{norm2form} that, for each fixed pair $(i,j)$,
\begin{equation}\label{esttoshow}
\abs{c(f^\ast(\epsilon^i \wedge \epsilon^j))\otimes f^\ast[\bar{c}(\epsilon_i)\bar{c}(\epsilon_j)]\phi(x)}\leqslant \abs{f^\ast[\bar{c}(\epsilon_i)\bar{c}(\epsilon_j)] \phi(x)}= \abs{\phi(x)},
\end{equation}
and therefore, for all $s\in[0,1]$,
\begin{equation}\label{estfs}
\inppro{c(R^{E}_s)\phi,\phi}\geqslant -(2k-1)(2k-2)s(1-s)\norm{\phi}^2
\geqslant -\frac{1}{4}(2k-1)(2k-2)\norm{\phi}^2.
\end{equation}
In particular, the right-most ``$\geqslant$" in \eqref{estfs} is strict unless $s=\frac{1}{2}$, and it follows immediately that, for all $s\neq \frac{1}{2}$ and any $\phi\in C^\infty(M,S\otimes E)$ such that $\norm{\phi}\neq 0$,
\begin{equation}\label{invertbileneq12}
\norm{D_s^E \phi}^2=\inppro{\de2 \phi, \phi}\geqslant \frac{1}{4}(2k-1)(2k-2)\norm{\phi}^2+\inppro{c(R^{E}_s)\phi,\phi}>0.
\end{equation}
Thus the operator $D_s^E$ is invertible for all $s\in [0,1]\setminus \{\frac{1}{2}\}$.

It now remains to deal with the case $s=\frac{1}{2}$. Assume now that $\phi \in \ker(D_{1/2}^E)$, then
\[
\inppro{(D_{1/2}^E)^2 \phi, \phi}=\norm{D_{1/2}^E \phi}^2=0.
\]
Therefore, by \eqref{lch}, one has that
\begin{equation}\label{curvest}
\frac{1}{4}\inppro{k^{TM}\phi,\phi}+\inppro{c(R^{E}_{1/2})\phi,\phi}=0.
\end{equation}
On the other hand, the left-hand side \eqref{curvest} is
\[
\inppro{(c(R^{E}_{1/2})+\frac{1}{4}k^{TM})\phi,\phi}=\int_M \inpro{(c(R^{E}_{1/2})+\frac{1}{4} k^{TM})\phi(x),\phi(x)} \d x.
\]
In addition, from the estimate \eqref{estfs}, it follows that for all $x\in M$,
\begin{equation}\label{curvest2}
\inpro{(c(R^{E}_{1/2})+\frac{1}{4} k^{TM})\phi(x),\phi(x)}\geqslant 0.
\end{equation}
The inequality \eqref{curvest2}, together with \eqref{curvest}, implies that
\[
\inpro{(c(R^{E}_{1/2})+\frac{1}{4} k^{TM})\phi(x),\phi(x)}\equiv 0.
\]
By the assumption that $k^{TM}(x)>(2k-1)(2k-2)$ at some $x\in M$, there exists an $\alpha>0$ and a neighborhood $U\subset M$ of $x$ such that $k^{TM}(y)\geqslant (2k-1)(2k-2)+4\alpha$ for all $y\in U$. Then, for all $y\in U$, 
\[
\alpha\abs{\phi(y)}\leqslant \inpro{(c(R^{E}_{1/2})+\frac{1}{4} k^{TM})\phi(y),\phi(y)}=0,
\]
and thus $\phi\vert_U\equiv 0$. It now follows from the unique continuation property of Dirac operators \cite[Chapter 8]{booss2012elliptic} that $\phi \equiv 0$. Therefore $D_{1/2}^E$ is invertible.
\end{proof}

Now we come back to show the following proposition, which plays a similar role of \cite[Lemma 4.5]{llarull1998sharp} in the proof. In particular, it shows straightforwardly how the $(1,\Lambda^2)$-contracting assumption implies the estimate \eqref{esttoshow} and will be helpful to adapt the method here to convex hypersurfaces in the next section. 

\begin{prop}\label{norm2form}
Let $\tilde{E}$ be any Hermitian vector bundle on $M$. We extend the action of the Clifford bundle $\mathbb{C}l(TM)$ on the spinor bundle $S$ to an action on $S\otimes \tilde{E}$ by acting as identity on $\tilde{E}$, and still denote it by $c(\cdot)$. Given any $2$-form $\alpha\in \Omega^2(M, \mathbb{R})$ such that $\alpha \wedge \alpha=0$, one has for all $\phi \in C^\infty(M,S\otimes \tilde{E})$ that
\begin{equation}\label{estclff2form}
\abs{c(\alpha) \phi(x)}=\abs{\alpha(x)}\abs{\phi(x)} \qquad \textit{for all } x \in M.
\end{equation}
\end{prop}

This proposition can be viewed as a corollary of \cite[Lemma 2]{listing2010scalar}. However, for the completeness, we present a detailed proof below. In this paper, the result is applied to the $2$-forms $f^\ast (\epsilon^i \wedge \epsilon^j)$, for which we have $f^\ast (\epsilon^i \wedge \epsilon^j)\wedge f^\ast (\epsilon^i \wedge \epsilon^j)=f^\ast (\epsilon^i \wedge \epsilon^j \wedge \epsilon^i \wedge \epsilon^j)=0$. 

\begin{proof}[Proof of Proposition \ref{norm2form}]
Given any $x\in M$, take an arbitrary orthonormal basis $\{e_i\}_{i=1}^{\dim M}$ of $T_x M$, and then we have
\[
\begin{split}
\inpro{\phi(x), c(\alpha) \phi(x)}=&\sum_{k<l} \inpro{\phi(x), \alpha(e_k, e_l)c(e_k)c(e_l) \phi(x)}=\sum_{k<l} \alpha(e_k, e_l) \inpro{c(e_l)c(e_k) \phi(x), \phi(x)}\\
=&\sum_{k<l} \alpha(e_k, e_l) \inpro{-c(e_k)c(e_l) \phi(x), \phi(x)}=-\inpro{c(\alpha) \phi(x), \phi(x)},
\end{split}
\]
where each $\alpha(e_k, e_l)\in \mathbb{R}$, and therefore
\begin{equation*}
\abs{c(\alpha) \phi(x)}^2=\inpro{c(\alpha) \phi(x), c(\alpha) \phi(x)}=-\inpro{c(\alpha)c(\alpha) \phi(x), \phi(x)}.
\end{equation*}
On the other hand, note that
\[
c(\alpha) c(\alpha)=\frac{1}{4}\sum_{k,l,p,q} \alpha(e_k, e_l)\alpha(e_p, e_q)c(e_k)c(e_l)c(e_p)c(e_q).
\]
And we separate the summation into three different types.

\textbf{Type I:} $k,l,p,q$ mutually distinct. The sum of all terms of this type is
\[
\begin{split}
\mathrm{I}=&\frac{1}{4}\sum_{i_1<i_2<i_3<i_4}\sum_{\sigma \in \Sigma_4}\alpha(e_{i_{\sigma(1)}}, e_{i_{\sigma(2)}})\alpha(e_{i_{\sigma(3)}}, e_{i_{\sigma(4)}})c(e_{i_{\sigma(1)}})c(e_{i_{\sigma(2)}})c(e_{i_{\sigma(3)}})
c(e_{i_{\sigma(4)}})\\
=&\frac{1}{4}\sum_{i_1<i_2<i_3<i_4} \sum_{\sigma \in \Sigma_4}\mathrm{sgn}(\sigma) \alpha(e_{i_{\sigma(1)}}, e_{i_{\sigma(2)}})\alpha(e_{i_{\sigma(3)}}, e_{i_{\sigma(4)}})c(e_{i_{1}})c(e_{i_{2}})c(e_{i_{3}})c(e_{i_{4}})\\
=&\sum_{i_1<i_2<i_3<i_4} [\alpha\wedge \alpha(e_{i_1},e_{i_2},e_{i_3},e_{i_4})]c(e_{i_{1}})c(e_{i_{2}})c(e_{i_{3}})c(e_{i_{4}})=0,
\end{split}
\]
because $\alpha \wedge \alpha=0$. Throughout this proof, $\Sigma_n$ denotes the permutation group of the set $\{1,2, \dots, n\}$. 

\textbf{Type II:} $\{p,q\}=\{k,l\}$. The sum of all terms of this type is
\[
\begin{split}
\mathrm{II}=&\frac{1}{4}\sum_{k,l}\alpha(e_k, e_l)c(e_k)c(e_l)[\alpha(e_k, e_l)c(e_k)c(e_l)+\alpha(e_l, e_k)c(e_l)c(e_k)]\\
=-&\frac{1}{2}\sum_{k,l}(\alpha(e_k, e_l))^2=-\abs{\alpha(x)}^2,
\end{split}
\]
where the pointwise norm $\abs{\cdot}$ of $2$-forms is induced by the inner-product \eqref{metric2form}.

\textbf{Type III:} For the rest of the terms, there is precisely one index among $p,q$ that coincides with $k$ or $l$, and the sum of all such terms can now be rewritten as
\[
\begin{split}
\mathrm{III}=\frac{1}{4}\sum_{i_1<i_2<i_3}\sum_{\sigma \in \Sigma_3} &\alpha(e_{i_{\sigma(1)}}, e_{i_{\sigma(2)}})\alpha(e_{i_{\sigma(3)}}, e_{i_{\sigma(1)}})c(e_{i_{\sigma(1)}})c(e_{i_{\sigma(2)}})c(e_{i_{\sigma(3)}})
c(e_{i_{\sigma(1)}})\\
+&\alpha(e_{i_{\sigma(1)}}, e_{i_{\sigma(2)}})\alpha(e_{i_{\sigma(1)}},e_{i_{\sigma(3)}} )c(e_{i_{\sigma(1)}})c(e_{i_{\sigma(2)}})c(e_{i_{\sigma(1)}}) c(e_{i_{\sigma(3)}})
\\
+&\alpha(e_{i_{\sigma(1)}}, e_{i_{\sigma(2)}})\alpha(e_{i_{\sigma(2)}}, e_{i_{\sigma(3)}})c(e_{i_{\sigma(1)}})c(e_{i_{\sigma(2)}})c(e_{i_{\sigma(2)}})
c(e_{i_{\sigma(3)}})\\
 +&\alpha(e_{i_{\sigma(1)}}, e_{i_{\sigma(2)}})\alpha(e_{i_{\sigma(2)}}, e_{i_{\sigma(3)}})c(e_{i_{\sigma(1)}})c(e_{i_{\sigma(2)}})c(e_{i_{\sigma(2)}}) c(e_{i_{\sigma(3)}}).
\end{split}
\]

Note that
\[
\begin{split}
\sum_{\sigma \in \Sigma_3} &\alpha(e_{i_{\sigma(1)}}, e_{i_{\sigma(2)}})\alpha(e_{i_{\sigma(3)}}, e_{i_{\sigma(1)}})c(e_{i_{\sigma(1)}})c(e_{i_{\sigma(2)}})c(e_{i_{\sigma(3)}})
c(e_{i_{\sigma(1)}})\\
=\frac{1}{2}\sum_{\sigma \in \Sigma_3} &\alpha(e_{i_{\sigma(1)}}, e_{i_{\sigma(2)}})\alpha(e_{i_{\sigma(3)}}, e_{i_{\sigma(1)}})c(e_{i_{\sigma(1)}})c(e_{i_{\sigma(2)}})c(e_{i_{\sigma(3)}})
c(e_{i_{\sigma(1)}})\\
+&\alpha(e_{i_{\sigma(3)}}, e_{i_{\sigma(1)}})\alpha(e_{i_{\sigma(1)}}, e_{i_{\sigma(2)}})c(e_{i_{\sigma(3)}})
c(e_{i_{\sigma(1)}})c(e_{i_{\sigma(1)}})c(e_{i_{\sigma(2)}})=0,
\end{split}
\]
which holds the same way for the sum of all other terms. Therefore, we have that $\mathrm{III}=0$.

It finally follows that
\begin{equation}
\begin{split}
\abs{c(\alpha) \phi(x)}^2=-\inpro{c(\alpha)c(\alpha) \phi(x), \phi(x)}=\abs{\alpha(x)}^2 \abs{\phi(x)}^2.
\end{split}
\end{equation}

\end{proof}

\subsection{Rigidity}
The last thing to show is the rigidity part of Theorem \ref{Llarull}. 
\begin{prop}\label{rgsphere}
If the scalar curvature $k^{TM}\equiv (2k-1)(2k-2)$, then $f$ has to be an isometry. 
\end{prop}

Llarull's method to prove the rigidity part for the even-dimensional case is actually sufficient for proving Proposition \ref{rgsphere}. However, we base our proof on the following algebraic result, so that our approach can be adapted naturally to the rigidity part of both Theorem \ref{Llarull} and Theorem \ref{convexsfc}, and might also be useful in other situations.

\begin{prop}\label{normeq}
Let $f: (M,g^{TM}) \to (N,g^{TN})$ be a $(1,\Lambda^2)$-contracting map, where $\dim M= \dim N=n\geqslant 3$. If at a point $x\in M$, there exists an orthonormal basis $\{\epsilon^i\}_{i=1}^n$ of $T^\ast_{f(x)} N$, such that $\abs{f^\ast (\epsilon^i \wedge \epsilon^j)}=1$ for all $1\leqslant i<j\leqslant n$, then $f^\ast$ maps $T^\ast_{f(x)} N$ isometrically to $T^\ast_x M$.
\end{prop}
\begin{proof}
We first show that $\{f^\ast (\epsilon^i \wedge \epsilon^j)\}_{1\leqslant i<j\leqslant n}$ forms an orthonormal basis of $(\Lambda^2(T^\ast_x M), g^{TM})$, i.e. $f^\ast$ determines an isometry from $(\Lambda^2(T^\ast_{f(x)} X),g^{TN})$ to $(\Lambda^2(T^\ast_x M),g^{TM})$. With the assumption that each $f^\ast(\epsilon^i \wedge \epsilon^j)$ is of unit length, it actually suffices to show that $\{f^\ast (\epsilon^i \wedge \epsilon^j)\}_{1\leqslant i<j\leqslant n}$ is an orthogonal set. 

If it were not, then, without loss of generality, one could find a pair of distinct multi-indices $(i,j)$ and $(k,l)$, such that
\[
g^{TM}(f^\ast(\epsilon^i \wedge \epsilon^j), f^\ast (\epsilon^k \wedge \epsilon^l)) > 0.
\]
Let $a \in (0,1)$, and consider the $2$-form $\alpha=a\cdot \epsilon^i \wedge \epsilon^j+\sqrt{1-a^2}\cdot \epsilon^k \wedge \epsilon^l$. It follows by direct computation that $\abs{\alpha}=1$, and that
\begin{equation*}
\begin{split}
\abs{f^\ast \alpha}^2=&g^{TM}(f^\ast \alpha, f^\ast \alpha)\\
=&a^2 \abs{f^\ast (\epsilon^i \wedge \epsilon^j)}^2+(1-a^2)\abs{f^\ast(\epsilon^k \wedge \epsilon^l)}^2+2a \sqrt{1-a^2} g^{TM}(f^\ast(\epsilon^i \wedge \epsilon^j), f^\ast (\epsilon^k \wedge \epsilon^l))>1,
\end{split}
\end{equation*}
which contradicts the assumption that $f$ is $(1,\Lambda^2)$-contracting.

It now remains to show that $\{f^\ast (\epsilon_i)\}_{i=1}^n$ is an orthonormal basis of $T^\ast_x M$, for which we start by verifying it is indeed an orthogonal set. Assuming that there exists a pair $i\neq j$ such that $g^{TM}(f^\ast \epsilon^i, f^\ast \epsilon^j)\neq 0$, then $f^\ast(\epsilon^j)$ has a unique decomposition
\[
f^\ast(\epsilon^j)=b\cdot f^\ast (\epsilon^i)+\tilde{\epsilon}^j.
\]
for some $b\neq 0$ and $\tilde{\epsilon}^j \perp f^\ast (\epsilon^i)$ being nonzero. Take another $k\neq i,j$, and decompose $f^\ast(\epsilon^k)$ the same way as
\[
f^\ast(\epsilon^k)=c \cdot f^\ast (\epsilon^i)+\tilde{\epsilon}^k,
\]
for some $c\in \mathbb{R}$ (can be $0$) and $\tilde{\epsilon}^k \perp f^\ast (\epsilon^i)$ being nonzero. And hence
\[
f^\ast (\epsilon^i \wedge \epsilon^k)=f^\ast \epsilon^i \wedge \tilde{\epsilon}^k.
\]

Since it has been shown that $g^{TM}(f^\ast(\epsilon^i \wedge \epsilon^j),f^{\ast}(\epsilon^i \wedge \epsilon^k))=0$, while on the other hand, by \eqref{metric2form},
\[
g^{TM}(f^\ast(\epsilon^i \wedge \epsilon^j),f^{\ast}(\epsilon^i \wedge \epsilon^k))= g^{TM}(f^\ast \epsilon^i \wedge \tilde{\epsilon}^j,f^{\ast}\epsilon^i \wedge \tilde{\epsilon}^k)=\abs{f^\ast (\epsilon^i)}^2 \cdot g^{TM}(\tilde{\epsilon}^j,\tilde{\epsilon}^k),
\]
where $\abs{f^\ast (\epsilon^i)}^2>0$, one then has that $g(\tilde{\epsilon}^j, \tilde{\epsilon}^k)=0$.
It then follows from \eqref{metric2form} that
\[
\begin{split}
g^{TM}(f^\ast \epsilon^i \wedge \tilde{\epsilon}^k, \tilde{\epsilon}^j \wedge f^\ast \epsilon^k)=&g^{TM}(f^\ast \epsilon^i \wedge \tilde{\epsilon}^k,\tilde{\epsilon}^j \wedge  \tilde{ \epsilon}^k)+c \cdot g^{TM}(f^\ast \epsilon^i \wedge \tilde{\epsilon}^k,\tilde{\epsilon}^j \wedge  f^\ast \epsilon^i)\\
=&\abs{\tilde{\epsilon}^k}^2\cdot g^{TM}(f^\ast \epsilon^i, \tilde{\epsilon}^j)-c \abs{f^\ast \epsilon^i}^2 \cdot g^{TM}(\tilde{\epsilon}^k ,\tilde{\epsilon}^j)=0.
\end{split}
\]
Noticing that
\[
f^\ast(\epsilon^j\wedge \epsilon^k)=b \cdot f^\ast (\epsilon^i \wedge \epsilon^k)+\tilde{\epsilon}^j \wedge f^\ast \epsilon^k,
\]
it now follows from direct computation that
\[
\begin{split}
g^{TM}(f^\ast (\epsilon^i \wedge \epsilon^k), f^\ast (\epsilon^j\wedge \epsilon^k))=& b \cdot g^{TM}(f^\ast (\epsilon^i \wedge \epsilon^k), f^\ast (\epsilon^i\wedge \epsilon^k))+g^{TM}(f^\ast \epsilon^i \wedge \tilde{\epsilon}^k, \tilde \epsilon^j\wedge f^\ast \epsilon^k)\\
=&b\cdot g^{TM}(f^\ast (\epsilon^i \wedge \epsilon^k), f^\ast (\epsilon^i\wedge \epsilon^k))\neq 0,
\end{split}
\]
which would contradict the result that $g^{TM}(f^\ast (\epsilon^i \wedge \epsilon^k), f^\ast (\epsilon^j\wedge \epsilon^k))=0$. This means that $\{f^\ast (\epsilon^i)\}_{i=1}^n$ forms an orthogonal set.

Now we assume that there exists an $f^\ast(\epsilon^i)$ such that $\abs{f^\ast(\epsilon^i)}\neq1$, and, without loss of generality, $\abs{f^\ast(\epsilon^i)}<1$. Then we take two distinct numbers $j,k\neq i$. Since $\abs{f^\ast(\epsilon^i\wedge \epsilon^{j})}=\abs{f^\ast(\epsilon^i\wedge \epsilon^{k})}=1$, one has that $\abs{f^\ast( \epsilon^{j})},\abs{f^\ast(\epsilon^{k})}>1$. And then,
\[
\abs{f^\ast( \epsilon^{j}\wedge \epsilon^{k})}>1,
\]
which contradicts the assumption that $\abs{f^\ast( \epsilon^{j}\wedge \epsilon^{k})}=1$. Therefore $\abs{f^\ast(\epsilon^i)}=1$ for all $i=1,2,\dots, n$.
\end{proof}

\begin{proof}[Proof of Proposition \ref{rgsphere}]
Assuming that $k^{TM} \equiv (2k-1)(2k-2)$, it then follows again from \eqref{invertbileneq12} that, for any $s\in [0,1]\setminus\{\frac{1}{2}\}$, $D^E_s$ is invertible. On the other hand, according to \eqref{sf}, the spectral flow $\sfl{D^E_s}\neq 0$, so $D^E_{1/2}$ cannot be invertible, i.e. there exists $\phi \in \ker(D^E_{1/2})$ such that $\norm{\phi}\neq 0$. Then $\phi$ has to satisfy that
\begin{equation}
\inppro{(c(R^{E}_{1/2})+\frac{1}{4}k^{TM})\phi,\phi}=\int_M \inpro{(c(R^{E}_{1/2})+\frac{1}{4} k^{TM})\phi(x),\phi(x)} \d x=0.
\end{equation}
Since by \eqref{cfsph} and \eqref{esttoshow}, for all $x\in M$,
\[
\inpro{(c(R^{E}_{1/2})+\frac{1}{4} k^{TM})\phi(x),\phi(x)}\geqslant 0,
\]
it then follows that
\begin{equation}\label{eqcvtw}
\inpro{c(R^{E}_{1/2})\phi(x),\phi(x)}=-\frac{1}{4} k^{TM}\abs{\phi(x)}^2
=-\frac{1}{4}(2k-1)(2k-2)\abs{\phi(x)}^2.
\end{equation}

On the other hand, given an $x \in M$ and an orthonormal basis $\{\epsilon_i\}_{i=1}^{2k-1}$ of $T_{f(x)}\Sp^{2k-1}$, let $\{\epsilon^i\}_{i=1}^{2k-1}$ be the dual basis of $\{\epsilon_i\}_{i=1}^{2k-1}$. Then, by \eqref{cfsph}, 
\[
c(R^{E}_{1/2})\phi(x)=-\frac{1}{2}\sum_{i<j}c(f^\ast(\epsilon^i \wedge \epsilon^j))\otimes f^\ast[\bar{c}(\epsilon_i)\bar{c}(\epsilon_j)]\phi(x),
\]
where, for each pair $(i,j)$, it follows from Proposition \ref{norm2form} that 
\[
\abs{\inpro{c(f^\ast(\epsilon^i \wedge \epsilon^j))\otimes f^\ast[\bar{c}(\epsilon_i)\bar{c}(\epsilon_j)] \phi(x),\phi(x)}}\leqslant \abs{\phi(x)}^2.
\]
Therefore, \eqref{eqcvtw} requires that, for each pair $(i,j)$,
\begin{equation}\label{equality}
\inpro{c(f^\ast(\epsilon^i \wedge \epsilon^j))\otimes f^\ast[\bar{c}(\epsilon_i)\bar{c}(\epsilon_j)] \phi(x),\phi(x)}=\abs{\phi(x)}^2,
\end{equation}
which implies that $\abs{c(f^\ast(\epsilon^i \wedge \epsilon^j))}=1$. Then, from Proposition \ref{norm2form}, it follows that $\abs{f^\ast(\epsilon^i \wedge \epsilon^j)}=1$ for all $1\leqslant i <j \leqslant 2k-1$. And according to Proposition \ref{normeq}, $f$ is a local isometry. Furthermore, since $\Sp^{2k-1}$ is simply connected, $f$ is automatically an isometry.
\end{proof}

\section{Proof of Theorem \ref{convexsfc}: generalization to convex closed hypersurfaces in $\mathbb{R}^{2k}$}

In this section, we replace the sphere $\Sp^{2k-1}$ with an arbitrary smooth strictly convex closed hypersurface $X\subset \mathbb{R}^{2k}$, and equip $X$ with the induced metric $g_0$. Denote by $k^{TX}$ the associated scalar curvature. Let $f: M \rightarrow X$ be a smooth $(1,\Lambda^2)$-contracting map of nonzero degree, such that for any $x\in M$
\begin{equation}\label{fcontracting}
k^{TX}(f(x)) \leqslant k^{TM}(x).
\end{equation}

Denote by $\mathbf{G}: X \to \Sp^{2k-1}\subset \mathbb{R}^{2k}$ the Gauss map. Once again, as in \cite[pp. 96]{petersen2006riemannian}, for each $x\in X$, we identify $T_x \mathbb{R}^{2k}$ with $\mathbb{R}^{2k}$, so that $\mathbf{G}(x)$ is identified with the outward unit normal vector of $X$ at the point $x$ and $T_x X$ is identified with $\{\mathbf{G}(x)\}^\perp \subset \mathbb{R}^{2k}$.  Since in Section 2.3, $\{\mathbf{G}(x)\}^\perp$ is also identified with $T_{\mathbf{G}(x)}\Sp^{2k-1}$, the shape operator defined by $\mathbf{S}=\d\mathbf{G}: T_x X \to T_{\mathbf{G}(x)} \Sp^{2k-1}$ is now an endomorphism of the space $\{\mathbf{G}(x)\}^\perp \cong T_x X \cong  T_{\mathbf{G}(x)}\Sp^{2k-1}$. 

\subsection{Construction of a vector bundle}
Once again, we consider the vector bundle $E_1=\mathbf{G}^\ast E_0$ with the pull-back metric, and equip it with a family of unitary connections $\{\nabla^{E_1}_s=\mathbf{G}^\ast \nabla_s=\mathbf{G}^\ast(\d+s\cdot\omega): s\in [0,1]\}$, where $\omega$ is defined in \eqref{cntform}. Then the associated curvature is 
\[
R^{E_1}_s=\mathbf{G}^\ast F_s, \qquad \forall s\in [0,1].
\]
At an arbitrary point $x\in X$, let $\{\epsilon_i\}_{i=1}^{2k-1}$ be an orthonormal basis of $T_x X$ and $\{\epsilon^i\}_{i=1}^{2k-1}$ be its dual basis. Since $T_x X$ is identified with $T_{\mathbf{G}(x)} \Sp^{2k-1}$, we also regard $\{\epsilon_i\}_{i=1}^{2k-1}$ and $\{\epsilon^i\}_{i=1}^{2k-1}$ as an orthonormal basis of $T_{\mathbf{G}(x)}\Sp^{2k-1}$ and $T^\ast_{\mathbf{G}(x)} \Sp^{2k-1}$ respectively. And thus, according to \eqref{curvfs}, we write that
\[
F_s(\mathbf{G}(x))=-2s(1-s)\sum_{k<l}\epsilon^k \wedge \epsilon^l \otimes \bar{c}(\epsilon_k)\bar{c}(\epsilon_l).
\]
And then, as in \eqref{rssph}, the curvature $R^{E_1}_s$ is determined by \footnote{
Although it might not be used in the proof, we would like to mention that the Gauss equation actually tells that $g_0(\mathbf{S}(\epsilon_i),\epsilon_k)g_0(\mathbf{S}(\epsilon_j),\epsilon_l)-g_0(\mathbf{S}(\epsilon_i),\epsilon_l)g_0(\mathbf{S}(\epsilon_j),\epsilon_k)=g_0(R^{TX}(\epsilon_i,\epsilon_j)\epsilon_k, \epsilon_l)$.}
\begin{equation*}
\begin{split}
R^{E_1}_s(x)(\epsilon_i,\epsilon_j)=&\mathbf{G}^\ast [F_s(\mathbf{G}(x))](\epsilon_i,\epsilon_j)
=-2s(1-s) \sum_{k< l}\mathbf{G}^\ast(\epsilon^k\wedge \epsilon^l) (\epsilon_i, \epsilon_j) \cdot \mathbf{G}^\ast[\bar{c}(\epsilon_k) \bar{c}(\epsilon_l)]\\
=&-2s(1-s) \sum_{k< l}(\epsilon^k\wedge \epsilon^l) (\mathbf{S}(\epsilon_i), \mathbf{S}(\epsilon_j)) \cdot \mathbf{G}^\ast[\bar{c}(\epsilon_k) \bar{c}(\epsilon_l)]\\
=&-2s(1-s) \sum_{k< l}\left[g_0(\mathbf{S}(\epsilon_i),\epsilon_k)g_0(\mathbf{S}(\epsilon_j),\epsilon_l)-g_0(\mathbf{S}(\epsilon_i),\epsilon_l)g_0(\mathbf{S}(\epsilon_j),\epsilon_k)\right]\cdot \mathbf{G}^\ast[\bar{c}(\epsilon_k) \bar{c}(\epsilon_l)].
\end{split}
\end{equation*}

Furthermore, since $X$ is strictly convex, the shape operator $\mathbf{S}$ is strictly positive, and therefore, the orthonormal basis $\{\epsilon_i\}_{i=1}^{2k-1}$ of $T_x X$ can be chosen so that $\mathbf{S}(\epsilon_i)=\lambda_i \epsilon_i$ for $\lambda_i>0$ ($i=1,2, \dots, 2k-1$). It then follows that
\begin{equation}
R^{E_1}_s(x)(\epsilon_i,\epsilon_j)=-2s(1-s)\lambda_i \lambda_j \mathbf{G}^\ast[\bar{c}(\epsilon_i)\bar{c}(\epsilon_j)].
\end{equation}
Then, as an analogue of \eqref{curvfs} , one has that
\begin{equation}\label{xscurv}
R^{E_1}_s(x)=-2s(1-s)\sum_{i<j}\lambda_i \lambda_j (\epsilon^i \wedge \epsilon^j) \otimes \mathbf{G}^\ast[\bar{c}(\epsilon_i)\bar{c}(\epsilon_j)].
\end{equation}

\subsection{A family of Dirac operators and its spectral flow}
Now on the manifold $M$, we denote by $S$ its spinor bundle $S(TM)$ and equip it with the connection $\nabla^S$ induced by the Levi-Civita connection on $TM$. Let $E=f^\ast E_1$ be the pull-back of the vector bundle $E_1$ equipped with the pull-back metric. Then $\{f^\ast \nabla^{E_1}_s: s\in [0,1]\}$ gives a family of unitary connections on $E$ with associated curvature operator $R^E_s=(f^\ast \nabla^{E_1}_s)^2=f^\ast R^{E_1}_s$.
We consider the twisted bundle $S\otimes E$ over the manifold $M$ and equip it with a family $\{\nabla^{S\otimes E}_s\}_{s\in [0,1]}$ of unitary connections determined by
\begin{equation}
\nabla^{S\otimes E}_s= \nabla^S\otimes 1+1\otimes f^\ast \nabla^{E_1}_s.
\end{equation}
This induces a family of twisted Dirac operators $D^E_s: C^\infty(M,S\otimes E) \rightarrow C^\infty(M,S\otimes E)$, which, in terms of a local orthonormal frame $\{e_i\}_{i=1}^{2k-1}$ of $TM$, is determined by
\begin{equation}
D^E_s=\sum_{i=1}^{2k-1} c(e_i) \nabla^{S\otimes E}_{s,e_i}.
\end{equation}

The computation of the spectral flow $\sfl{D_s^E}$ then follows as a simple generalization of the result in the previous section.
\begin{prop}
The spectral flow $\sfl{D_s^E}$ is given by
\begin{equation}
\sfl{D^E_s}=-\deg(f)\neq 0.
\end{equation}
\end{prop}
\begin{proof}
For the spectral flow of the family $\{D^E_s: s\in [0,1]\}$ here, we consider the map $\mathbf{G}\circ f: M \rightarrow \Sp^{2k-1}\subset \mathbb{R}^{2k}$, and define $g_M \in C^\infty(M, \End(E))$ by $g_M=(\mathbf{G}\circ f)^\ast g$. Then, for each $s\in[0,1]$,
\[
D_s^E=D_0^E+s g_M^{-1} [D_0^E, g_M].
\]
It then follows from \cite[Theorem 2.8]{getzler1993odd}, in precisely the same way as in Section 3.2, that
\begin{equation}
\sfl{D^E_s}=-\deg(\mathbf{G}\circ f)=-\deg(\mathbf{G})\deg(f).
\end{equation}
On the other hand, since the hypersurface $X$ is strictly convex, the Gauss map $\mathbf{G}$ is a diffeomorphism, and therefore, $\deg(\mathbf{G})=1$. Combining with the assumption that $\deg(f)\neq 0$, it implies that
\begin{equation}
\sfl{D^E_s}=-\deg(f)\neq 0.
\end{equation}
\end{proof}

\subsection{Vanishing of the spectral flow}
Similarly, in this case, we show again that the geometric assumptions imply the invertibility of the operator $D^E_s$ for every $s\in [0,1]$. 
\begin{thm}
Assume that the scalar curvature $k^{TM}\geqslant f^\ast (k^{TX})$ all over the manifold $M$, and $k^{TM}(x)> k^{TX}(f(x))$ at some $x\in M$, then the operator $D^E_s$ is invertible for every $s\in [0,1]$.
\end{thm}
\begin{proof}
The proof is pretty much the same as in the previous section. From the Lichnerowicz formula, it follows that, for any $\phi \in C^\infty (S\otimes E)$,
\begin{equation}
\inppro{\de2 \phi, \phi}=\inppro{\nabla^{S\otimes E,\ast}_s \nabla^{S\otimes E}_s \phi,\phi}+\frac{1}{4}\inppro{k^{TM}\phi,\phi}+\inppro{c(R^{E}_s)\phi,\phi}.
\end{equation}

Given any $x\in M$, we take an orthonormal basis of $\{\epsilon_i\}_{i=1}^{2k-1}$ of $T_{f(x)}X$, such that $\mathbf{S}(\epsilon_i)=\lambda_i \epsilon_i$ ($\lambda_i>0$), and let $\{\epsilon^i\}_{i=1}^{2k-1}$ be its dual basis. It follows from \eqref{xscurv} that
\begin{equation}\label{twcvcnvx}
\begin{split}
c(R^{E}_s)\phi(x)=&c((\mathbf{G}\circ f)^\ast F_s)\phi(x)=c(f^\ast(\mathbf{G}^\ast F_s) )\phi(x)\\
=&-2s(1-s)\sum_{\substack{i<j}}\lambda_i \lambda_j c(f^\ast(\epsilon^i \wedge \epsilon^j))\otimes (\mathbf{G}\circ f)^\ast [\bar{c}(\epsilon_i)\bar{c}(\epsilon_j)]\phi(x).
\end{split}
\end{equation}
Once again, it follows from Proposition \ref{norm2form} that, for each $(i,j)$,
\begin{equation}\label{estshown}
\abs{\inpro{c(f^\ast(\epsilon^i\wedge\epsilon^j))\otimes (\mathbf{G}\circ f)^\ast [\bar{c}(\epsilon_i)\bar{c}(\epsilon_j)]\phi(x),\phi(x)}}\leqslant \abs{\phi(x)}^2.
\end{equation}
So, at the point $x\in M$, one has for all $s\in[0,1]$ that
\begin{equation}\label{cvxtwtcurv}
\begin{split}
\inpro{c( R^{E}_s)\phi(x),\phi(x)}=&-2s(1-s)\sum_{i<j}\lambda_i \lambda_j\inpro{c(f^\ast(\epsilon^i \wedge \epsilon^j))\otimes (\mathbf{G}\circ f)^\ast[\bar{c}(\epsilon_i)\bar{c}(\epsilon_j)]\phi(x), \phi(x)}\\
                                            \geqslant &-2s(1-s)\sum_{i<j}\lambda_i \lambda_j\abs{\phi(x)}^2\\
=&-s(1-s)k^{TX}(f(x))\abs{\phi(x)}^2.
\end{split}
\end{equation}

The rest of this proof is precisely the same as that of Theorem \ref{invsph}. It follows that the operator $D^E_s$ is invertible for all $s\in [0,1]$.
\end{proof}

\subsection{Rigidity}
\begin{prop}
If the scalar curvature $k^{TM}\equiv f^\ast k^{TX}$, then $f$ has to be an isometry. 
\end{prop}
\begin{proof}
This proof is, to a large extent, the same as Proposition \ref{rgsphere}. Assuming $k^{TM} \equiv f^\ast k^{TX}$, it then follows again from the above computation that, $D^E_s$ is invertible for all $s\in [0,1]\setminus \{\frac{1}{2}\}$. On the other hand, since $\sfl{D^E_s}\neq 0$, $D^E_{1/2}$ cannot be invertible, i.e. there exists a $\phi\in \ker(D^E_{1/2})$, such that $\norm{\phi}\neq 0$. Then $\phi$ has to satisfy that
\begin{equation}\label{vanish}
\inppro{(c(R^E_{1/2})+\frac{1}{4}k^{TM})\phi,\phi}=\int_M \inpro{(c(R^E_{1/2})+\frac{1}{4} k^{TM})\phi(x),\phi(x)} \d x=0.
\end{equation}

Given any $x\in M$, we take an orthonormal basis $\{\epsilon_i\}_{i=1}^{2k-1}$ of $T_{f(x) }X$ and its dual basis $\{\epsilon^i\}_{i=1}^{2k-1}$ as in the previous section, so that \eqref{twcvcnvx} holds. The formulae \eqref{estshown}, \eqref{cvxtwtcurv}, and \eqref{vanish} then imply that for each pair $(i,j)$,
\begin{equation}\label{equalitycnv}
\inpro{c(f^\ast(\epsilon^i\wedge\epsilon^j))\otimes (\mathbf{G}\circ f)^\ast [\bar{c}(\epsilon_i)\bar{c}(\epsilon_j)]\phi,\phi}(x)= \abs{\phi(x)}^2.
\end{equation}
Thus $\abs{f^\ast(\epsilon^i\wedge\epsilon^j)}=1$ for all $(i,j)$, and it then follows from Proposition \ref{normeq} that $f$ has to be a local isometry. In addition, $X$ is simply connected, because, for the smooth strictly convex closed hypersurface $X\subset \mathbb{R}^{2k}$, the Gauss map $\mathbf{G}$ is a diffeomorphism from $X$ to the unit sphere $\Sp^{2k-1}$. Therefore, $f$ is indeed an isometry.
\end{proof}


\textbf{Acknowledgments.}
The authors would like to thank Professor Weiping Zhang for helpful discussions and thank the anonymous referees for careful reading and suggestions. Y. Li was partially supported by Nankai Zhide Foundation. G. Su was partially supported by NSFC 12271266, NSFC 11931007, Nankai Zhide Foundation, and the Fundamental Research Funds for the Central Universities, No. 100-63233103. X. Wang was partially supported by NSFC Grant 12101361, the Project of Young Scholars of SDU, and the Fundamental Research Funds of Shandong University, Grant No. 2020GN063.





\end{document}